\documentclass[a4paper, 11pt]{amsart}

\usepackage[T1]{fontenc}
\usepackage[utf8]{inputenc}

\usepackage{amssymb}

\makeatletter
\@namedef{subjclassname@2020}{%
  \textup{2020} Mathematics Subject Classification}
\makeatother

\allowdisplaybreaks[2]

\newtheorem{theorem}{Theorem}[section]
\newtheorem{proposition}[theorem]{Proposition}
\newtheorem{lemma}[theorem]{Lemma}
\newtheorem{corollary}[theorem]{Corollary}
\theoremstyle{definition}
\newtheorem{remark}[theorem]{Remark}
\newtheorem{definition}[theorem]{Definition}
\newtheorem{example}[theorem]{Example}

\numberwithin{equation}{section}

\begin{document}

\title[Two-sided zero  product determined algebras]{Two-sided zero  product determined algebras}

\thanks{Supported by the Slovenian Research Agency (ARRS) Grant P1-0288. }

\author{\v Zan Bajuk}
\address{\v Z. B., Faculty of Mathematics and Physics,  University of Ljubljana, Slovenia}
\email{zanbajuk@gmail.com}

\author{Matej Bre\v sar} 
\address{M. B., Faculty of Mathematics and Physics,  University of Ljubljana,  and Faculty of Natural Sciences and Mathematics, University
of Maribor, Slovenia}
\email{matej.bresar@fmf.uni-lj.si}


\begin{abstract}   An algebra $A$ is said to be  two-sided zero product determined  if every   bilinear functional $\varphi:A\times A\to F$
satisfying $ \varphi(x,y)=0$ whenever $xy=yx=0$ 
is of the form 
$\varphi(x,y)=\tau_1(xy) + \tau_2(yx)$ for some linear functionals $\tau_1,\tau_2$ on $A$. We present  some basic properties and equivalent definitions,  examine connections with some properties of derivations, and as the main result prove that 
a finite-dimensional simple algebra  that is not a division algebra  is   two-sided zero product determined  if and only if it is separable. 
\end{abstract}

\newcommand\E{\ell}
\newcommand\mathcalM{{\mathcal M}}
\newcommand\pc{\mathfrak{c}}

\newcommand{\enp}{\begin{flushright} $\Box$ \end{flushright}}

\maketitle

\section{Introduction} 

Let $F$ be an arbitrary  field. 
  By an algebra, we will always mean an associative $F$-algebra with unity $1$. The restriction to unital algebras is not absolutely necessary, but will make our exposition simpler.  We will also assume that all algebra bimodules are unital.
  
 An algebra $A$ is said to be {\em zero product determined} ({\em zpd} for short)  if every   bilinear functional $\varphi:A\times A\to F$ with the property that for all $x,y\in A$, $xy=0$ implies 
  $\varphi(x,y)=0$, is of the form $\varphi(x,y)=\tau(xy) $ for some linear functional $\tau$. These algebras are the subject of the  recent book \cite{zpdbook}.
  
  This paper is devoted to the following related notion.
  
  \begin{definition}
  An algebra $A$ is  {\em two-sided zero product determined} ({\em 2-zpd} for short) if every   bilinear functional $\varphi:A\times A\to F$ with the property that for all $x,y\in A$,
\begin{equation} \label{2-zpd} xy=yx=0\implies \varphi(x,y)=0,\end{equation}
is of the form 
\begin{equation} \label{ta}\varphi(x,y)=\tau_1(xy) + \tau_2(yx)\end{equation} for some linear functionals $\tau_1,\tau_2$ on $A$. 
  \end{definition}
  
These algebras have indirectly appeared in the consideration of zpd algebras, but have, so far, played an important role only in the functional analytic branch of the theory  \cite[Chapter 6]{zpdbook}. Besides the result by  Kosan,  Lee, and  Zhou
\cite{KLZ} which states that a finite-dimensional central simple algebra that is not a division algebra is  2-zpd,
 not much is known  in the purely algebraic context.
Our aim is to fill this gap.

In Section \ref{s2} we study formal properties of 2-zpd algebras. In particular, we 
present some equivalent definitions and discuss the relation between 2-zpd,  zpd, and  {\em zero Lie product determined}  algebras ({\em zLpd} for short). The latter  are algebras $A$ such that every   bilinear functional $\varphi:A\times A\to F$ with the property that for all $x,y\in A$, $ [x,y]=0$ implies 
  $\varphi(x,y)=0$, is of the form $\varphi(x,y)=\tau([x,y]) $ for some linear functional $\tau$. Here, as usual, $[x,y]$ stands for $xy-yx$.


In Section \ref{s3} we will see that the condition that an algebra $A$ is 2-zpd is closely related to some properties of derivations of $A$, more specifically, with the innerness of derivations from $A$ to its dual $A^*$ and with the property that the images of derivations lie in the linear span of square-zero elements. 

Using some of the results in Sections \ref{s2} and \ref{s3}, we prove in  Section \ref{s4} that a finite-dimensional simple algebra  that is not a division algebra  is 2-zpd   if and only if it is separable. This generalizes  the aforementioned result from \cite{KLZ} and  is the main result of the paper. 
It was discovered rather unexpectedly since 
the definition 
of a 2-zpd algebra is apparently entirely unrelated  to the definition of 
 the classical notion of a 
separable algebra.

We will not discuss applications in this short paper. Some indications of the usefulness of the concept of a 2-zpd algebra are given in \cite[Section 7.1]{zpdbook} and \cite{BGV}, and we hope to provide more evidence in the
future.

\section{Introductory remarks}\label{s2}

 
We begin by presenting the basic connection between zpd, zLpd, and 2-zpd algebras.

\begin{proposition}\label{zpzp}
A 2-zpd algebra $A$ is  zpd and zLpd.  
\end{proposition}

\begin{proof}
Let
 $\varphi:A\times A\to F$ be a bilinear functional.
 
Suppose  $\varphi(x,y)=0$ whenever $xy=0$. In particular, $\varphi(x,y)=0$ whenever $xy=yx=0$, so
 $\varphi(x,y)=\tau_1(xy)+\tau_2(yx)$ for some linear functionals $\tau_1$ and $\tau_2$. This  yields $\varphi(x,1)=\varphi(1,x)$, which is enough to conclude that $A$ is zpd \cite[Proposition 2.7]{zpdbook}. 
 
 Suppose now that $\varphi(x,y)=0$ whenever $[x,y]=0$. Since $xy=yx=0$ implies $xy=yx$, we have, as above,
 $\varphi(x,y)=\tau_1(xy)+\tau_2(yx)$. Further, $[x,1] =0$, hence $\varphi(x,1)=0$, and therefore $\tau_2=-\tau_1$. Thus,  $\varphi(x,y)=\tau_1([x,y])$, which shows that $A$ is zLpd. 
\end{proof}

Example \ref{ezz} below shows that the converse is not true, i.e., an algebra $A$ that is zpd and zLpd may not be 2-zpd. On the other hand, if $A$ is commutative,
 then there is obviously no difference between the conditions that $A $ is zpd and $A$ is 2-zpd. 
 Since every algebra that is generated by idempotents is zpd \cite[Theorem 2.15]{zpdbook}, the following proposition holds.
 
 \begin{proposition} If a commutative algebra $A$  is generated by idempotents, then $A$ is 2-zpd.
 \end{proposition}
 
Let us  remark that the class of commutative algebras that are generated by idempotents is quite special, see \cite{K}.


We continue by stating  \cite[Theorem 3.23]{zpdbook} which provides a closer relation between the zpd and 2-zpd conditions.

\begin{theorem}\label{txyzw}
 Let $A$ be a  zpd algebra.   If a  bilinear functional $\varphi:A\times A\to F$ satisfies  \eqref{2-zpd}, then
\begin{equation} \label{xyzwa}
\varphi(xy,zw)+\varphi(wx,yz)=\varphi(x,yzw) + \varphi(wxy,z) 
\end{equation}
for all $x,y,z,w\in A$.
\end{theorem}

\begin{remark}
If $\varphi$ is symmetric and the characteristic of $F$ is not $2$, then by taking $x=z=1$ in  \eqref{xyzwa} we obtain $$\varphi(y,w) = \frac{1}{2}\varphi(yw+wy,1),$$
 which is a definitive conclusion. 
\end{remark}

Our next goal is to give some equivalent definitions of 2-zpd algebras. 
We first remark that an equivalent way of stating \eqref{ta} is that there exists a linear functional $\tau$ such that
\begin{equation} \label{ta1}\varphi(x,y)-\varphi(xy,1) =\tau([x,y])\end{equation}
for all $x,y\in A$.
Indeed, \eqref{ta} implies \eqref{ta1} with $\tau=-\tau_2$, and the converse is obvious. 

For any algebra $A$, we write
$$(A\otimes A)_0={\rm span}\, \{x\otimes y\,|\,xy=yx=0\}\subseteq A\otimes A,$$
where span stands for the linear span.
As usual, by $[A,A]$ we denote the linear span of all commutators in $A$.

\begin{theorem} \label{teq}The following conditions are equivalent for an algebra $A$:
\begin{enumerate}
\item[{\rm (i)}] $A$ is 2-zpd.
\item[{\rm (ii)}]  If  $\Phi:A\times A\to X$, where $X$ is any vector space over $F$, is a bilinear map with the property that for all $x,y\in A$, $ xy=yx=0$ implies $\Phi(x,y)=0$, then there exist
a linear map $T:[A,A]\to X$ such that $$\Phi(x,y)-\Phi(xy,1) =T([x,y])\quad \mbox{for all $x,y\in A$}.$$
\item[{\rm (iii)}] $A$ is zLpd and for all $z,w\in A$, \begin{equation*} \label{tu} [z,w]=0\implies z\otimes w - zw\otimes 1 \in (A\otimes A)_0.\end{equation*}
\end{enumerate}
\end{theorem}

\begin{proof} (i)$\implies$(ii). Let $\Phi:A\times A\to X$ be a bilinear map  such that $ xy=yx=0$ implies $\Phi(x,y)=0$. Take a linear functional $\omega$ on $X$ and observe that the assumption that $A$ is 2-zpd is applicable to the bilinear functional $\omega\circ \Phi$. Thus, $\omega\circ \Phi$ satisfies \eqref{ta1}, i.e., there exists a linear functional $\tau_\omega$ such that
 $$\omega\bigl (\Phi(x,y)-\Phi(xy,1)\bigr) =\tau_\omega([x,y])$$ for all $x,y\in A$.
 Hence, if  $x_i,y_i$ are any elements in $A$ such that $\sum_i [x_i,y_i]=0$, then
  $$\sum_i \omega\bigl (\Phi(x_i,y_i)-\Phi(x_iy_i,1)\bigr)=0$$
 for every linear functional $\omega$. Therefore, we actually have
   $$\sum_i\Phi(x_i,y_i)-\Phi(x_iy_i,1)=0.$$
   Accordingly, the map 
$T:[A,A]\to X$ given by
$$T\Big( \sum_i [x_i,y_i] \Big)= \sum_i\Phi(x_i,y_i)-\Phi(x_iy_i,1)$$
is well-defined. It is obviously 
linear and so  (ii) is proved.

(ii)$\implies$(iii). Proposition \ref{zpzp} tells us that  $A$ is zLpd if (ii) holds.   Define $\Phi:A\times A\to (A\otimes A)/(A\otimes A)_0$ by 
$$\Phi(x,y) = x\otimes y + (A\otimes A)_0.$$
Clearly, $xy=yx=0$ implies $\Phi(x,y)=0$. Hence, by (ii), for any pair of commuting elements $z,w$ we have  $\Phi(z,w)-\Phi(zw,1)=0$, meaning that  $z\otimes w - zw\otimes 1 \in  (A\otimes A)_0$.

(iii)$\implies$(i).  Let $\varphi:A\times A\to F$ be a bilinear functional satisfying $\varphi(x,y)=0$ whenever $xy=yx=0$. Observe that (iii) implies that for all $z,w\in A$,
$$[z,w]=0\implies 
\varphi(z,w)-\varphi(zw,1)=0.$$ Therefore, we may use the assumption that $A$ is zLpd to the bilinear functional $(z,w)\mapsto \varphi(z,w)-\varphi(zw,1),$ and thus conclude that \eqref{ta1} holds. 
\end{proof}

\begin{remark}
It is easy to see that an algebra $A$ is zLpd if and only if for all $z_i,w_i\in A$ such that $\sum_i [z_i,w_i]=0$, we have
$$\sum_i z_i\otimes w_i \in {\rm span}\,\{x\otimes y\,|\, [x,y]=0\}$$
(see \cite[Proposition 1.3]{zpdbook}). 
Condition (iii) can thus be fully expressed in terms of tensor products.
\end{remark}

We conclude this section by showing that 2-zpd algebras behave well with respect to (finite) direct products.

\begin{proposition}\label{pd}
Algebras $A_1$ and $A_2$ are 2-zpd if and only if  their direct product $A=A_1 \times A_2$ is 2-zpd. 
\end{proposition}

\begin{proof}
Suppose $A_1$ and $A_2$ are 2-zpd and let   $\varphi:A\times A\to F$ be a bilinear functional such that for all $x,y\in A$,
$xy=yx=0$ implies $\varphi(x,y)=0$. Define a bilinear functional  $\varphi_1:A_1\times A_1\to F$
by $$\varphi_1(x_1,y_1)=\varphi\big((x_1,0),(y_1,0)\big).$$
Observe that $x_1y_1=y_1x_1=0$ implies $\varphi_1(x_1,y_1)=0$. Since $A_1$ is 2-zpd, there exist linear functionals $\tau_{11}$ and $\tau_{12}$ on $A_1$ such that
$$\varphi\big((x_1,0),(y_1,0)\big)=\varphi_1(x_1,y_1)=\tau_{11}(x_1y_1) + \tau_{12}(y_1x_1)$$
for all $x_1,y_1\in A_1$. Similarly we see that there exist linear functionals $\tau_{21}$ and $\tau_{22}$ on $A_2$ such that
$$\varphi\big((0,x_2),(0,y_2)\big)=\tau_{21}(x_2y_2) + \tau_{22}(y_2x_2)$$
for all $x_2,y_2\in A_2$. Define linear functionals $\tau_1$ and $\tau_2$ on $A$ by
$$\tau_1\big((x_1,x_2)\big)=\tau_{11}(x_1) + \tau_{21}(x_2)\quad\mbox{and}\quad \tau_2\big((x_1,x_2)\big)=\tau_{12}(x_1) + \tau_{22}(x_2).$$
Take arbitrary $(x_1,x_2),(y_1,y_2)\in A$.
Observe that our assumption on $\varphi$ implies that  $$\varphi\big((x_1,0), (0,y_2)\big) = \varphi\big((0,x_2), (y_1,0)\big)  =0.$$
Hence,
 \begin{align*}\varphi\big((x_1,x_2), (y_1,y_2)\big) =& \varphi\big((x_1,0), (y_1,0)\big)  + \varphi\big((0,x_2), (0,y_2)\big)  \\
 =& \tau_{11}(x_1y_1) + \tau_{12}(y_1x_1) + \tau_{21}(x_2y_2) + \tau_{22}(y_2x_2)\\
 =&\tau_1\big((x_1,x_2)(y_1,y_2)\big) + \tau_2\big((y_1,y_2)(x_1,x_2)\big),
 \end{align*}
which proves that $A$ is 2-zpd.

Assume now that $A$ is 2-zpd and take a bilinear functional $\varphi_1$ on $A_1$ satisfying $\varphi_1(x_1,y_1)=0$ whenever $x_1y_1=y_1x_1=0$.  Define $\varphi:A\times A\to F$ by $$\varphi\big((x_1,x_2), (y_1,y_2)\big) = \varphi_1(x_1,y_1).$$
Observe that $\varphi$ is bilinear and satisfies  $\varphi(x,y)=0$ whenever $xy=yx=0$. Therefore, there are linear functionals $\tau_1$ and $\tau_2$ on $A$ such that $\varphi(x,y)=\tau_1(xy) + \tau_2(yx)$ for all $x,y\in A$. Hence,
$$\varphi_1(x_1,y_1) = \varphi\big((x_1,0), (y_1,0)\big) = \tau_1\big((x_1y_1,0)\big) + \tau_2\big((y_1x_1,0)\big),$$
which proves that $A_1$ is 2-zpd. Similarly we see that $A_2$ is 2-zpd.
\end{proof}

The direct product of infinitely many commutative zpd algebras does not need to be zpd \cite[Example 2.24]{zpdbook}. Since the notions of zpd and 2-zpd algebras coincide in the commutative setting, this shows that Proposition \ref{pd} does not hold for infinite direct products.

\section{Connections with derivations}\label{s3}

Let $A$ be an algebra and $M$ be an $A$-bimodule. 
Recall that a linear 
map $\delta:A\to M$ is called a {\em derivation} if $\delta(xy)=\delta(x)y + x\delta(y)$ for all $x,y\in A$. Any fixed element $m\in M$ gives rise to a derivation defined by $x\mapsto xm -mx $. Such a derivation is said to be {\em inner}.

Let $A^*$ denote the dual space of $A$.  It is well known that  $A^*$ becomes an $A$-bimodule by defining 
$$(x\cdot f)(y) =f(yx)\quad\mbox{and}\quad (f\cdot x)(y) =f(xy)$$
for $f\in A^*$, $x,y\in A$. We can therefore consider derivations from $A$ to $A^*$. 

The following theorem is 
an algebraic analog of \cite[Theorem 6.5]{zpdbook}.
\begin{theorem}\label{ccc}
Let $A$ be a zpd algebra.  If every derivation from $A$ to $A^*$ is inner, then $A$ is 2-zpd.
\end{theorem}

\begin{proof}
Let $\varphi:A\times A\to F$ be a bilinear functional   satisfying  \eqref{2-zpd}. 
 Define $\psi:A\times A\to F$ by
$$\psi(x,y)=\varphi(x,y) - \varphi(xy,1).$$
Then 
\begin{align}
\psi(xy,w)=\varphi(xy,w)-\varphi(xyw,1),\nonumber\\
\label{xyzwb} \psi(wx,y)=\varphi(wx,y)-\varphi(wxy,1),\\
\psi(x,yw) = \varphi(x,yw) - \varphi(xyw,1).\nonumber
\end{align}
By Theorem \ref{txyzw}, $\varphi$ satisfies \eqref{xyzwa}. In particular, 
\begin{equation*} 
\varphi(xy,w)+\varphi(wx,y) =\varphi(x,yw) + \varphi(wxy,1),
\end{equation*}
which together with \eqref{xyzwb} gives
\begin{equation}\label{der}\psi(x,yw) = \psi(wx,y) + \psi(xy,w)\end{equation}
for all $x,y,w\in A$.

Define $\delta:A\to A^*$
by
$$\delta(y)(x)=\psi(x,y).$$
Using \eqref{der}, we obtain
\begin{align*}\delta(yw)(x)&= \psi(x,yw)= \psi(wx,y) + \psi(xy,w) \\
&= \delta(y)(wx) + \delta(w)(xy) \\&= \big(\delta(y)\cdot w + y\cdot \delta(w)\big)(x),\end{align*}
which shows that $\delta$ is a derivation. By our assumption, there exists a $\tau\in A^*$ such that $\delta(y)=y\cdot \tau-\tau\cdot y $ for all $y\in A$. This means that $\psi(x,y)=\tau(xy-xy)$ for all $x,y\in A$, and so $\varphi$ satisfies  \eqref{ta1}.
\end{proof}

The condition that a derivation from $A$ to $A ^*$ is inner can be equivalently stated as that the  Hochschild cohomology group $H^1(A, A^*)$ is trivial, which in turn is equivalent to the triviality of the 
Hochschild homology group $H_1(A, A)$. Every finite-dimensional algebra of finite global dimension has this property \cite[Proposition 6]{Han}.
A simple concrete  example is the algebra $T_n(F)$ of all upper triangular matrices over $F$, which is also zpd \cite[Example 2.19]{zpdbook}. 
The following corollary therefore holds.

\begin{corollary} The algebra $T_n(F)$ is 2-zpd.\end{corollary}

Theorem \ref{ccc} gives a sufficient  
condition for an algebra $A$ to be 2-zpd. The next theorem  presents a necessary condition
which  also involves derivations, but is 
 of a different nature.





\begin{theorem}\label{we}
Let $A$ be an algebra, let $M$ be an $A$-bimodule, and let  $$\Theta = {\rm span}\, \{am\,|\, a\in A, m\in M, ama=0\}.$$
If $A$ is 2-zpd, then  every derivation $\delta:A\to M$ satisfies 
$\delta(A)\subseteq \Theta$.
\end{theorem}

\begin{proof}Take $w\in A$. By Theorem \ref{teq}\,(iii), there exist $x_i,y_i\in A$ such that $x_iy_i=y_ix_i=0$ for every $i$ and
\begin{equation}\label{too}1\otimes w - w\otimes 1=\sum_i x_i\otimes y_i.\end{equation}
Observe that $\delta(1)=0$. Therefore,
applying ${\rm id}_A \otimes \delta$ to \eqref{too} we obtain 
$$1\otimes \delta(w)=\sum_i x_i\otimes \delta(y_i),$$
and so, in particular, 
$$\delta(w)=\sum_i x_i\delta(y_i).$$
From $x_iy_i=0$ it follows that $\delta(x_i)y_i+x_i\delta(y_i)=0$, which along with $y_ix_i=0$ gives $x_i\delta(y_i)x_i=0$. Therefore, $x_i\delta(y_i)\in\Theta$, and hence $\delta(w)\in\Theta$.
\end{proof}

\begin{corollary}\label{me}
If $A$ is a 2-zpd algebra, then every derivation $\delta: A\to A$ has image in the linear span of elements whose square is $0$.
\end{corollary}

\begin{proof}Observe that $ama=0$ implies $(am)^2=0$.
\end{proof}

\begin{remark}
More generally, every derivation from a 2-zpd algebra $A$ 
to an algebra $M\supseteq A$ has image in  the linear span of square-zero elements in $M$. We have restricted to derivations from $A$ to $A$ since this  will be needed in the proof of Theorem \ref{mt}.\end{remark}

\begin{remark}
If $A$ is a zpd algebra, then every commutator in $A$ is a sum of square-zero elements \cite[Theorem 9.1]{zpdbook}. That is to say, the image of every inner derivation from $A$ to $A$ lies in the linear span of square-zero elements. Corollary \ref{me} states that if $A$ is 2-zpd, then this holds for all, not only inner,  derivations of $A$. In the proof of Theorem \ref{mt} we will see that this is not true if $A$ is only assumed to be zpd.  
\end{remark}

\begin{remark} We can say more 
 about inner derivations from a 2-zpd algebra $A$ to itself. Denoting by $L_x$ (resp.\ $R_y$) the left (resp.\ right) multiplication map, that is, $L_x(u)= xu$ and $R_y(u)=uy$, then every inner derivation $u\mapsto [u,w]$ can be written as $\sum_i L_{x_i}R_{y_i}$ where $x_i,y_i\in A$ are such that $x_iy_i=y_ix_i=0$ for every $i$. This follows immediately from \eqref{too}. If $A$ is only zpd, we can write every inner derivation as $\sum_i L_{x_i}R_{y_i}$ where either
 $x_iy_i=0$ for every $i$ or $y_ix_i=0$ for every $i$ \cite[Remark 9.3]{zpdbook}. 
\end{remark}

We can now show that the converse to Proposition \ref{zpzp} is not true. The following example is a refinement of \cite[Example 3.28]{zpdbook}.

\begin{example} \label{ezz}  Let
 $A_0$ be any algebra that is not equal to the linear span of its square-zero elements, and let $C$ be an algebra generated by a nilpotent element (i.e.,  $C\cong F[X]/(X^n)$ for some $n\ge 2$).
 We claim that the algebra
$$A=A_0\otimes C$$
is not 2-zpd. To prove this, note first that every element $a\in A$ can be uniquely written as $a=a_0+ a_1 u+\dots + a_{n-1}u^{n-1}$ where $a_i\in A_0 (=A_0\otimes 1)$ and $u\in C$ is a  nilpotent element of index $n$. Consider the ideal $M=Au$ of $A$ as an $A$-bimodule and define
$\delta:A \to M$ by $$\delta\left(\sum_{k=0}^{n-1}a_ku^k\right) = \sum_{k=1}^{n-1} ka_ku^k.$$
 Observe that $\delta$ is a derivation from $A$ to $M$. Take $a_1 \in A_0$ that does not lie in the linear span of square-zero elements of $A_0$. If $A$ was 2-zpd, then by Theorem \ref{we} there would exist $x_{ik},  y_{i\ell}\in A_0$ such that
$$\delta(a_1u)=a_1 u = \sum_i \left(\sum_{k=0}^{n-1} x_{ik} u ^k\right) \left(\sum_{\ell=1}^{n-1} y_{i\ell} u ^\ell \right) $$ and 
$$\left(\sum_{k=0}^{n-1} x_{ik} u ^k\right)   \left(\sum_{\ell=1}^{n-1} y_{i\ell} u ^\ell \right)  \left(\sum_{k=0}^{n-1} x_{ik} u ^k\right) =0$$
for every $i$. Observe that this yields $a_1 =  \sum_i x_{i0}y_{i1}$ and $x_{i0}y_{i1} x_{i0} =0$ for every $i$. The latter condition implies that each $x_{i0}y_{i1}$ has square zero, which contradicts our assumption on $a_1$.

Taking the matrix algebra $M_m(F)$, $m\ge 2$, for $A_0$ we see that the algebra $M_m(C)$ is not 2-zpd. However, this algebra is both 
zpd and zLpd \cite[Corollary 2.17\,(a) and Corollary 3.11]{zpdbook}. 
\end{example}

\section{Finite-dimensional  simple algebras}\label{s4}


Recall that an algebra $A$  said to be {\em separable} if $A\otimes A$ contains an element $\sum_i p_i\otimes q_i$ such that 
$\sum_i p_iq_i =1$ and 
$\sum xp_i\otimes q_i = \sum p_i\otimes q_ix$ for every $x\in A$. This is equivalent to the condition that $A$ is  the direct product of finitely many  matrix algebras
$M_n(D)$, where $n\ge 1$ and $D$ is a  
finite-dimensional division algebra  whose center  is a finite separable field extension of $F$. 
Thus, a separable algebra is finite-dimensional and semisimple, and if $F$ is a perfect field,  the converse is also true.
 Since central simple algebras are separable,  the following theorem generalizes the main result of \cite{KLZ}.




\begin{theorem} \label{mt}Let $A$ be a  finite-dimensional simple algebra that is not a division algebra. Then $A$ is 2-zpd if and only if $A$ is separable.
\end{theorem}

\begin{proof}  We first remark that,  by Wedderburn's theorem,  $A$ is isomorphic to the matrix algebra $M_n(D)$ where $n\ge 2$ and $D$ a division algebra, and is hence  a  zpd algebra \cite[Corollary 2.17\,(a)]{zpdbook}.

Let $A$ be separable. Then
every derivation from  $A$ to an $A$-bimodule $M$ is inner \cite[Theorem 8.1.19]{Ford}, and so
Theorem \ref{ccc} implies that $A$ is 2-zpd. 
However, let us also prove this directly,  avoiding derivations. Thus,
  take a  bilinear functional $\varphi:A\times A\to F$ satisfying \eqref{2-zpd}. Theorem \ref{txyzw} tells us that
$\varphi$ satisfies   \eqref{xyzwa}.
Writing $1$ for $x$, $p_i$ for $y$, $q_iy$ for $z$, and $x$ for $w$ in \eqref{xyzwa}, where
$\sum_i p_i\otimes q_i$ is the element from the definition of a separable algebra, 
 we obtain
\begin{equation}\label{zra}\varphi(xp_i,q_iy) = \varphi(p_i,q_iyx)+ \varphi(x,p_iq_i y) - \varphi(1,p_iq_iyx)\end{equation}
 for all $x,y\in A$. Next, from
$$\sum_i xp_i \otimes q_i y
= \Big(\sum_i xp_i\otimes q_i\Big)(1\otimes y) 
=  \Big(\sum_i p_i\otimes q_ix\Big)(1\otimes y) 
= \sum_i p_i\otimes q_ixy
$$
it follows that
$$\sum_i \varphi(xp_i,q_iy) = \sum_i \varphi(p_i,q_ixy)$$
for all $x,y\in A$. Together with \eqref{zra} this gives
$$\sum_i \varphi(p_i,q_iyx)+ \varphi(x,p_iq_i y) - \varphi(1,p_iq_iyx)= \sum_i  \varphi(p_i,q_ixy)$$
Using 
 $\sum_i p_iq_i=1$ we obtain
 $$\varphi(x,y)= \sum_i \varphi(p_i,q_ixy)  -  \sum_i \varphi(p_i,q_iyx) +\varphi(1,yx),$$
 which shows that $\varphi$ is of the desired form \eqref{ta}. Thus, $A$ is 2-zpd.

To prove the converse, assume
that $A$ is not separable.  Then there exists a nonzero $F$-linear derivation $\delta$ of the center $K$ of $A$ \cite[Corollary 11.5.4]{cohn}.
Since $x\mapsto a\delta(x)$ is still a derivation for any $a\in K$, we may assume that $\delta(z) = 1$ for some $z\in K$. 
By \cite[Theorem 6]{H} (or \cite[Theorem]{amitsur}) we can extend $\delta$ to a derivation of $A$. Now, 
a theorem by Wedderburn   tells us that  a finite-dimensional algebra is not equal to the linear span of its nilpotent elements  (see, e.g., \cite[Theorem 3.21]{DF}). In particular, the linear span of all square-zero elements, which we denote by $N_A$, is a proper subspace of $A$. 
Take $b\in A\setminus N_A$.
From 
$$\delta(zb) - z\delta(b) = b \notin N_A$$
we see that at least one of the elements $\delta(zb)$ and $\delta(b)$ does not lie in $N_A$. Thus, 
$A$ contains an element $a$ such that
$\delta(a)\notin N_A$, and so   Corollary \ref{me}  tells us that $A$ is not 2-zpd.
\end{proof}


\begin{remark}
In light of Proposition \ref{pd},  we can state Theorem \ref{mt} in the following slightly more general form: If $A$ is a finite-dimensional semisimple algebra such that none of its simple components is a division algebra, then $A$ is
is 2-zpd if and only if $A$ is separable.
\end{remark}

\begin{remark}
If $A$ is as in Theorem \ref{mt}, then the linear span of square-zero elements $N_A$ is equal to $[A,A]$. Indeed, since $A$ is zpd, we have
 $[A,A]\subseteq N_A$  \cite[Theorem 9.1]{zpdbook}.  Proving the converse is easy: as $A$ is von Neumann regular, for every $a\in A$ there exists a $b\in A$ such that
$a=aba$, and hence $a= [ab,a]\in [A,A]$ if $a^2 =0$ (alternatively, one can use the fact that if viewing $A$ as an algebra over its center, 
$[A,A]$ has codimension $1$ in $A$ 
 \cite[Exercise 4.12]{INCA}, and hence it cannot be properly contained in $N_A$). The  proof of Theorem \ref{mt} thus shows that the condition that $A$ is 2-zpd  is also equivalent to the condition that every  derivation from $A$ to $A$ has  image  in $[A,A]$.
\end{remark}

\noindent
{\bf Acknowledgment.} The authors are thankful to Professor Claude Cibils for drawing their attention to \cite{Han}.

\end{document}